\documentclass{amsart}
\numberwithin{equation}{section}

\newtheorem{thm}{Theorem}[section]
\newtheorem{lem}[thm]{Lemma}
\newtheorem{cor}[thm]{Corollary}

\newtheorem{defin}[thm]{Definition}
\newtheorem{rem}[thm]{Remark}

\begin{document}

\begin{center}
\textbf{{\large {Backward and non-local problems for the Rayleigh-Stokes equation}}}\\[0pt]
\medskip \textbf{Ravshan Ashurov$^{1}$ and Nafosat Vaisova $^{2}$}\\[0pt]
\medskip \textit{\ $^{1}$ Institute of Mathematics, Academy of Science of Uzbekistan}

\textit{ashurovr@gmail.com\\[0pt]
$^{2}$ Institute of Mathematics, Uzbekistan Academy of Science, Urgench}

\end{center}

\textbf{Abstract}: The Fourier method is used to find conditions on the right-hand side and on the initial data in the Rayleigh-Stokes problem, which ensure the existence and uniqueness of the solution. Then, in the Rayleigh-Stokes problem, instead of the initial condition, consider the non-local condition: $u(x,T)=\beta u(x,0)+\varphi(x)$, where $\beta$ is either zero or one. It is well known that if $\beta=0$, then the corresponding problem, called the backward problem, is ill-posed in the sense of Hadamard, i.e. a small change in $u(x,T)$ leads to large changes in the initial
	data. Nevertheless, we will show that if we consider sufficiently smooth current information, then
	the solution exists and it is unique and stable. It will also be shown that if $\beta=1$, then the corresponding non-local problem is well-posed and coercive type inequalities are valid.


\section{Introduction}

\textbf{1.} The main object of study in this paper is the following Rayleigh-Stokes problem for a generalized second-grade fluid with a time-fractional derivative model:
\begin{equation}\label{probIN}
	\left\{
	\begin{aligned}
		&\partial_t u(x,t)  -(1+\gamma\, \partial_t^\alpha)\Delta u(x,t) = f(x, t),\quad x\in \Omega, \quad 0< t \leq T;\\
		&u(x, t) = 0, \quad x\in \partial \Omega, \quad 0 < t \leq T;\\
		&u(x, 0)= \varphi(x), \quad x\in \Omega,
	\end{aligned}
	\right.
\end{equation}
where $\gamma>0$ is a fixed constant, $\varphi$ is the initial data, $\partial_t= \partial/\partial t$, and $\partial_t^\alpha$ is the Riemann-Liouville fractional derivative of the order $\alpha \in (0,1)$ defined by (see, e.g. \cite{KilSriTru}):
\begin{equation}\label{RL}
\partial_t^\alpha h(t)= \frac{d}{dt} \int\limits_0^t \omega_{1-\alpha}(t-s) h(s)ds, \quad \omega_{\alpha}(t)=\frac{t^{\alpha-1}}{\Gamma(\alpha)}.
\end{equation} 
 Here $\Gamma(\sigma)$ is
Euler's gamma function. Usually problem (\ref{probIN}) is considered in the domain $\Omega \subset R^N$, $N=1,2,3$, with a sufficiently smooth boundary $\partial \Omega$.

The Rayleigh-Stokes problem (\ref{probIN}) has received considerable attention in recent years because of its practical importance (see, e.g. \cite{Tan1}, \cite{Tan2}). This non-stationary flow problem investigates the propagation of a vortex in a half-space filled with a viscous in-compressible fluid that sets in motion when an infinite flat plate suddenly acquires a constant velocity parallel to itself from rest.
The fractional derivative $\partial_t^\alpha$
in the model is used to capture the viscoelastic behavior of
the flow (see e.g. \cite{Fet}, \cite{Shen}).

In order to get an idea of the  behavior of the solution of this model, a number of authors have shown considerable interest in obtaining an exact solution in some special cases; see e.g. \cite{Fet}, \cite{Shen}, \cite{Zhao}. For example, Shen et al. \cite{Shen} obtained an exact solution to the problem using the sine Fourier transform and then applying the fractional Laplace transform.  Zhao and Yang \cite{Zhao} obtained exact solutions using the eigenfunction expansion in a rectangular domain for the case of homogeneous initial and boundary conditions. In this case, the eigenfunctions are explicitly written out. The solutions obtained in these studies are of a formal nature, and the regularity of the solution has not been particularly studied.

Questions of the regularity of the solution were studied in the fundamental work Bazhlekova et al. \cite{Bazh}. The authors proved the Sobolev regularity of the homogeneous problem
both for smooth and non-smooth initial data $\varphi(x)$, including $\varphi(x) \in L_2(\Omega)$.

The exact solutions obtained in the above studies involve infinite series and special functions, such as generalized Mittag-Leffler functions, and are therefore inconvenient for numerical imple-mentation. In addition, explicit solutions are available only for a limited class of problem settings. Therefore, it is extremely important to develop efficient and optimally accurate numerical algorithms for problem (\ref{probIN}). Numerous works of specialists are devoted to this issue. An overview of the work in this direction is contained in the above work \cite{Bazh}. 

The inverse problems of determining the right-hand side (the heat source density) of
the Rayleigh-Stokes equation have been considered by a number of authors (see, e.g.,
\cite{Tran1}, \cite{Tran2}, \cite{Duc} and the bibliography therein). However, there is no general closed theory for the abstract case of the source function $F(x, t)$. Known results deal with the separated source term $F(x, t) = h(t)f(x)$, where $h(t)$ is a known function and $f(x)$ is an unknown one. Since this inverse problem is an ill-posed problem in the sense of Hadamard, various regularization methods are proposed in the above papers. Also in these papers, the proposed regularized methods were tested by simple numerical experiments to check the error estimate between the desired solution and the regularized solution.

Note that such an inverse problem is ill-posed in the Hadamard sense also for the subdiffusion equation (see, e.g. \cite{Kirane_f_1}, \cite{Kirane_f_2}, \cite{AshMuk}).

\textbf{2.} In the case when the initial condition $u(x, 0)= \varphi(x)$ in the problem (\ref{probIN})  is replaced by $u(x, T)= \varphi(x)$, then the resulting problem is called \emph{backward problem}. The backward problem for the Rayleigh-Stokes equation is of great importance and is aimed at determining the previous state of a physical field (for
example, at t = 0) based on its current information (see, e.g., \cite{Luc1}, \cite{Luc2} and the bibliography therein). However, this problem (as well as the inverse problem of finding the right-hand side of the equation) is ill-posed in the sense of Hadamard. In other words, a small change in $u(x,T)$ leads to a large change in the original data. The authors of the above works proposed various regularization methods and tested these methods using numerical experiments.

In the present paper we will pay special attention to the backward, since in previous papers (see, e.g., \cite{Luc1}, \cite{Luc2}) the authors considered only the case $N\leq 3$. And this is connected with the method used in these works: if the dimension of the space is less than four, then for the eigenvalues $\lambda_k$ of the Laplace operator with the Direchlet condition, the series $\sum_k \lambda_k^{-2}$ converges.

It is well known that in most models described by differential equations, an initial condition is used to select a single solution. However, there are also processes where we have to use non-local conditions, for example, the integral over time intervals, or connection of solution values at different times, for example, at the initial time and at the final time. It should be noted that non-local conditions model some details of natural phenomena more accurately, since they take into account additional information in the initial conditions.

In this work we consider problem (\ref{probIN}) with a non-local time condition: 
\[
u(x,T)= u(x,0)+\varphi(x), 
\]
It turns out that the non-local problem is well-posed. In other words, a solution to a non-local problem exists and is unique. Moreover, the solution depends continuously on the function $\varphi(x)$ in the non-local condition.

Thus, the problem of determining the previous state of the physical field on the basis of current information is an ill-posed problem. However, if, without knowing the initial and present states, we want them to differ by the value $\varphi(x)$, then such a problem turns out to be correct.
\section{Problem statements}

The method of our research is the Fourier method. Therefore, only eigenfunctions and eigenvalues are needed from the Laplace operator in the Rayleigh-Stokes problem (\ref{probIN}). With this in mind, instead of the Laplace operator $(-\Delta)$ in problem (\ref{probIN}), we consider an abstract positive operator.

So in a Hilbert space $H$ with the scalar product
$(\cdot, \cdot)$ and the norm $||\cdot||$, consider an arbitrary unbounded positive selfadjoint operator $A$. Let $A$ have a complete in $H$ system of orthonormal
eigenfunctions $\{v_k\}$ and a countable set of positive
eigenvalues $\lambda_k:$ $0<\lambda_1\leq\lambda_2 \cdot\cdot\cdot\rightarrow +\infty$. 
For a vector-valued functions (or simply functions)
$h: \mathbb{R}_+\rightarrow H$, we define the Riemann-Liouville fractional derivative of order $0<\alpha< 1$ in the same way as (\ref{RL}) (see, e.g. \cite{Liz})
Finally, let $C((a,b); H)$ stand
for a set of continuous functions $u(t)$  of $t\in (a,b)$ with
values in $H$.

Consider the following Rayleigh-Stokes problem
\begin{equation}\label{probMain}
	\left\{
	\begin{aligned}
		&\partial_t u(t)  + (1+\gamma\, \partial_t^\alpha)A u(t) = f(t),\quad 0< t \leq T;\\
		&u(T)=\beta u(0)+ \varphi,
	\end{aligned}
	\right.
\end{equation}
where $\gamma>0$ is a fixed constant, $\varphi\in H$ and $f(t) \in C((0,T]; H)$ and $\beta$ is equal to $0$ or $1$. If $\beta=0$ then this problem is called \textit{the backward problem} and if $\beta=1$, then \emph{the non-local
problem}. We will consider both of these cases.

In this article, various initial problems are considered. The solution of these problems is defined in a standard way: all vector functions and their derivatives included in the equation and in the initial conditions must be continuous in the variable $t$, and all the corresponding equalities should be understood as the equality of the vectors from $H$ at each point $t$. As a example, let us define a solution for the problem (\ref{probMain}).
\begin{defin}\label{def} A function $u(t)\in C([0,T]; H)$ with the properties
	$\partial_t u(t), \partial_t^\alpha Au(t)\in C((0,T); H)$ and satisfying conditions (\ref{probMain})  is called \textbf{the 	solution} of the nonlocal Rayleigh-Stokes  problem (\ref{probMain}).
\end{defin}

To solve problem (\ref{probMain}) with $\beta=1$, we divide it into two auxiliary
problems:
\begin{equation}\label{prob1a}
	\left\{
	\begin{aligned}
		&\partial_t v(t)  + (1+\gamma\, \partial_t^\alpha)A v(t) = f(t),\quad 0< t \leq T;\\
		&v(0) =0
	\end{aligned}
	\right.
\end{equation}
and
\begin{equation}\label{prob1b}
	\left\{
	\begin{aligned}
		&\partial_t w(t) + (1+\gamma\, \partial_t^\alpha)A w(t) = 0,\quad 0< t \leq T;\\
		&w(T) = w(0)+\psi, \quad 0 < \xi \leq T,
	\end{aligned}
	\right.
\end{equation}
where $\psi\in H$ is a given function.

Problem (\ref{prob1a}) and (\ref{prob1b}) are a special case of problem (\ref{probMain}), and solutions to problems (\ref{prob1a}) and (\ref{prob1b}) are defined similarly to Definition  \ref{def}.

If $\psi=\varphi-v(\xi)$ and $v(t)$ and $w(t)$ are the
corresponding solutions, then it is easy to verify that function
$u(t)=v(t)+w(t)$ is a solution to problem (\ref{probMain}).
Therefore, it is sufficient to solve the auxiliary problems.

\begin{rem}\label{laplace}
	We note that, as operator $A$ one can take, for example, the Laplace operator with the Dirichlet condition in an arbitrary $N$ (not only $\leq 3$) - dimensional domain with a sufficiently smooth boundary. For our reasoning, it is sufficient that the operator $A$ has the properties listed above.
\end{rem}

In \cite{AF2022} and \cite{AF_1_2022}, a similar non-local problem with an arbitrary parameter $\beta$ was studied in detail for subdiffusion equations with Riemann-Liouville and Caputo derivatives. For the classical diffusion equation with the parameter $\beta=1$, it was first considered in \cite{AshSob} and \cite{AshSob1}.

The remainder of this paper is composed of eight sections and the Conclusion. In the
next section, we introduce the Hilbert space associated with the degree of operator A and
recall some properties of function $B_\alpha(\lambda, t)$, introduced in \cite{Bazh}. Section 4 is devoted to the study
of the main problem (\ref{probIN}) with operator $A$ instead of the Laplace operator. Here, conditions are given on the right-hand side of the equation and on the initial function, under which the solution of problem (\ref{probIN}) is found in the form of a Fourier series. In Section 5, we study the backward
problem, i.e. problem (\ref{probMain}) with $\beta =0$. This problem is not well-posed according to Hadamard: a small change $u(T)$ leads to large changes in the solution of the problem. Nevertheless, this section shows that if we consider $u(T)$ in the class of smooth functions, then stability is preserved. Section 6 studies the conditional stability of the backward problem. The next two sections are devoted to the study of
the non-local problem (\ref{probMain}) with $\beta=1$.

\section{Preliminaries}

In this section, we introduce the Hilbert space of $"$smooth$"$ functions defined by the degree of operator $A$ and recall some properties of function $B_\alpha (\lambda, t)$ obtained in \cite{Bazh} and \cite{Luc1} that we will use in what follows.

Let $ \tau $ be an arbitrary real number. We introduce the power
of operator $ A $, acting in $ H $ according to the rule (note, operator $A$ is positive and therefore $\lambda_k>0$ for all $k$)
$$
A^\tau h= \sum\limits_{k=1}^\infty \lambda_k^\tau h_k v_k.
$$
Here and everywhere below, for the vector $h\in H$, the symbol $h_k$ will denote the Fourier coefficients of this vector: $h_k=(h, v_k)$. The domain of definition of this operator is determined from condition $A^\tau h\in H$ and has the form:
$$
D(A^\tau)=\{h\in H:  \sum\limits_{k=1}^\infty \lambda_k^{2\tau}
|h_k|^2 < \infty\}.
$$
For elements of $D(A^\tau)$ we introduce the norm
\[
||h||^2_\tau=\sum\limits_{k=1}^\infty \lambda_k^{2\tau} |h_k|^2 =
||A^\tau h||^2.
\]
With this norm the lenear-vector space $D(A^\tau)$ turns into a Hilbert
space.

Let $B_\alpha(\lambda, t)$ be the solution of the following Cauchy problem 
\begin{equation}\label{B}
	L y(t)\equiv	y'(t)+\lambda (1+\gamma \partial_t^\alpha)y(t)=0, \,\, t>0,\,\,\lambda>0,\,\, y(0)=1.
\end{equation}
The solution of such an equation is expressed in terms of the generalized Wright function (see e.g. A.A. Kilbas et. al \cite{KilSriTru}, Example 5.3, p. 289). But function $B_\alpha(\lambda, t)$, the solution of this Cauchy problem, is studied in detail in  Bazhlekova, Jin, Lazarov, and Zhou \cite{Bazh}. See also  Luc, Tuan, Kirane, Thanh \cite{Luc1}, where very important lower bounds are obtained. The authors of \cite{Bazh}, in particular, proved the following lemma.

\begin{lem}\label{Bazh} Let $B_\alpha(\lambda, t)$ be a solution of the Cauchy problem (\ref{B}). Then
	\begin{enumerate}
		\item
		$B_\alpha (\lambda, 0)=1,\,\, 0<B_\alpha (\lambda, t)<1,\,\, t>0$,
		\item$\partial_t B_\alpha (\lambda, t)<0, \,\, t\geq 0$,
		\item$\lambda B_\alpha (\lambda, t)< C \min \{t^{-1}, t^{\alpha-1}\}, \,\,t>0\}$,
		\item
		$\int\limits_0^T B_\alpha (\lambda, t) dt \leq \frac{1}{\lambda}, \,\, T>0.$
	\end{enumerate}
	
	\end{lem}

It should only be noted that in paper \cite{Bazh}, instead of assertion (2), a more general proposition was proved. But the assertion (2) in our lemma easily follows from the representation of the function $B_\alpha (\lambda, t)$, also obtained in \cite{Bazh}:
\begin{equation}\label{BInt}
B_\alpha (\lambda, t)=\int\limits_0^\infty e^{-rt} b_\alpha(\lambda, r) dr,
\end{equation}
where
\[
b_\alpha(\lambda, r)=\frac{\gamma}{\pi} \frac{\lambda r^\alpha \sin \alpha \pi}{(-r+\lambda\gamma r^\alpha \cos \alpha \pi +\lambda)^2+(\lambda \gamma r^\alpha \sin \alpha \pi )^2}.
\]

The following assertion is also implicitly contained in \cite{Bazh}. But in view of its importance in our reasoning, we present a proof.

\begin{lem}\label{BazhEquation} The solution of the Cauchy problem
	\begin{equation}\label{BCauchy}
		y'(t)+\lambda (1+\gamma \partial_t^\alpha)y(t)=f(t), \,\, t>0,\,\,\lambda>0,\,\, y(0)=0,
	\end{equation}
	has the form
	\begin{equation}\label{BCauchySolution}
	y(t)=\int\limits_0^t B_\alpha(\lambda, t-\tau) f(\tau) d \tau.	
	\end{equation}
\end{lem}

\begin{proof} Since $B_\alpha(\lambda, 0)=1$, then
	\[
	\partial_t y(t)= f(t) + \int\limits_0^t\partial_t B_\alpha(\lambda, t-\tau) f(\tau) d\tau.
	\]
On the other hand
\[
\partial_t^\alpha y(t)=\frac{1}{\Gamma(1-\alpha)}\partial_t \int\limits_0^t (t-\tau)^{-\alpha} \int\limits_0^\tau B_\alpha(\lambda, \tau-\xi) f(\xi) d\xi d \tau=
\]
(change the order of integration)
\[
=\frac{1}{\Gamma(1-\alpha)}\partial_t \int\limits_0^t f(\xi)  \int\limits_\xi^t (t-\tau)^{-\alpha}B_\alpha(\lambda, \tau-\xi) d\tau d\xi=
\]
(change of variables: $\tau-\xi=\eta,\, \tau=\xi+\eta,\, d\tau=d\eta$)
\[
=\frac{1}{\Gamma(1-\alpha)}\partial_t \int\limits_0^t f(\xi)  \int\limits_0^{t-\xi} (t-(\xi+\eta))^{-\alpha} B_\alpha(\lambda, \eta) d \eta d \xi=
\]	
\[
=\int\limits_0^t f(\xi)\frac{1}{\Gamma(1-\alpha)}\partial_t \int\limits_0^{t-\xi} (t-(\xi+\eta))^{-\alpha} B_\alpha(\lambda, \eta) d \eta d \xi=\int\limits_0^t f(\xi)\partial_t^\alpha B_\alpha(\lambda, t-\xi) d\xi.
\]
Thus, given that $B_\alpha(\lambda, t)$ is a solution to the Cauchy problem (\ref{B}), we obtain
\[
L y(t)= f(t) + \int\limits_0^t f(\xi) L B_\alpha(\lambda, \cdot-\xi) d\xi = f(t).
\]
	\end{proof}

\begin{cor}\label{BazhEquationIN} The solution of the Cauchy problem
	\begin{equation}\label{BCauchyIN}
		y'(t)+\lambda (1+\gamma \partial_t^\alpha)y(t)=f(t), \,\, t>0,\,\,\lambda>0,\,\, y(0)=y_0,
	\end{equation}
	has the form
	\begin{equation}\label{BCauchySolutionIN}
		y(t)=y_0 B_\alpha(\lambda, t) + \int\limits_0^t B_\alpha(\lambda, t-\tau) f(\tau) d \tau.	
	\end{equation}
\end{cor}

We also need the following properties of function $B_\alpha (\lambda, t)$.

\begin{lem}\label{B'}
	There is a constant $C>0$, such that
	\[
	|\partial_t B_\alpha (\lambda, t)| \leq \frac{C}{\lambda\,t^{2-\alpha} },\,\, t>0.
	\]
\end{lem}
\begin{proof} From (\ref{BInt}) we have
	\[
	\partial_t B_\alpha (\lambda, t)=-\int\limits_0^\infty r e^{-rt} b_\alpha(\lambda, r) dr.
	\]
	Therefore, by the definition of $b_\alpha$,
	\[
	|\partial_t B_\alpha (\lambda, t)|\leq \frac{\gamma}{\pi} \int\limits_0^\infty 
	\frac{\lambda r^\alpha \sin \alpha \pi}{(\lambda \gamma r^\alpha \sin \alpha \pi )^2}\,r  e^{-rt}  dr= \frac{1}{\gamma\pi\lambda \sin \alpha \pi } \int\limits_0^\infty r^{1-\alpha} e^{-rt} dr=
	\]	
	(change of variables: $\tau=rt,\, d \tau= tdr$)
	\[
	=\frac{t^{\alpha-2}}{\gamma\pi\lambda \sin \alpha \pi } \int\limits_0^\infty \tau^{1-\alpha} e^{-\tau} d\tau=  \frac{t^{\alpha-2}}{\gamma\pi\lambda \sin \alpha \pi } \Gamma(2-\alpha)= \frac{C}{\lambda\,t^{2-\alpha} }.
	\]
	
\end{proof}

One can also obtain an estimate for $\partial_t B_\alpha$ with a smaller singularity at $t=0$.

\begin{lem}\label{B'New}
	There is a constant $C>0$, such that
	\[
	|\partial_t B_\alpha (\lambda, t)| \leq C \frac{  \lambda}{t^{\alpha} },\,\, t>0.
	\]
\end{lem}
\begin{proof} Again by the definition of $b_\alpha$ one can obtain
	\[
	|\partial_t B_\alpha (\lambda, t)|\leq \frac{\gamma}{\pi} \int\limits_0^\infty 
	\frac{\lambda r^\alpha \sin \alpha \pi}{r^2}\,r  e^{-rt}  dr=\frac{\lambda \,\gamma\, \Gamma (\alpha)\,\sin \alpha \pi }{\pi t^\alpha}.
	\]

	\end{proof}

Now, by combining these two estimates, we can obtain the following statement, which we also apply further.

\begin{lem}\label{B'New_n}
	For any integer $n= 2^j,\, j=1,2,\cdots,$ there is a constant $C_n>0$, such that
	\[
	|\partial_t B_\alpha (\lambda, t)| \leq \frac{C_n  \lambda^{\frac{1}{n}}}{t^{1-\frac{1-\alpha}{n}} },\,\, t>0.
	\]
\end{lem}
\begin{proof} Apply Lemmas \ref{B'} and \ref{B'New} to get
	\begin{equation}\label{Bn}
	|\partial_t B_\alpha (\lambda, t)|^2 \leq C\frac{1}{\lambda t^{2-\alpha}} \cdot \frac{ \lambda}{t^\alpha}= \frac{C}{t^2},\,\, t>0.
	\end{equation}
	This estimate and Lemma \ref{B'New} imply
		\[
	|\partial_t B_\alpha (\lambda, t)|^2 \leq C\frac{1}{ t}\cdot \frac{ \lambda}{t^\alpha}= \frac{C \lambda}{t^{1+\alpha}},\,\, t>0,
	\]
	or
	\[
	|\partial_t B_\alpha (\lambda, t)| \leq  \frac{C \lambda^{\frac{1}{2}}}{t^{1-\frac{1-\alpha}{2}}},\,\, t>0.
	\]
	Now, using this estimate and (\ref{Bn}), we obtain
	\[
	|\partial_t B_\alpha (\lambda, t)| \leq  \frac{C \lambda^{\frac{1}{4}}}{t^{1-\frac{1-\alpha}{4}}},\,\, t>0.
	\]
	
	Repeating this process gives the assertion of the lemma.
	
	\end{proof}

In the future, instead of $\lambda$, we will have eigenvalues $\lambda_k$ of operator $A$. The next lower bound for $B_\alpha (\lambda_k, t)$ was obtained in  Luc, N.H., Tuan,
N.H., Kirane, M., Thanh, D.D.X \cite{Luc1}. For the convenience of the reader, we present this proof.
\begin{lem}\label{BLower}The following estimate holds for all $t \in [0, T]$ and $k\geq 1$:
	\[
	B_\alpha (\lambda_k, t)\geq \frac{C(\alpha, \gamma, \lambda_1)}{\lambda_k},
	\]
	where 
	\[
	C(\alpha, \gamma, \lambda_1)=\frac{\gamma \sin \alpha \pi}{4}\int\limits_0^\infty \frac{r^\alpha e^{-rT}}{\frac{r^2}{\lambda_1^2}+ \gamma^2 r^{2\alpha}+1} dr.
	\]
		\end{lem}
\begin{proof} Since the modulo of trigonometric functions does not exceed one and $(a+b+c)^2\leq 3(a^2+b^2+c^2)$, then 
	\[
(-r+\lambda_k\gamma r^\alpha \cos \alpha \pi +\lambda_k)^2+(\lambda_k \gamma r^\alpha \sin \alpha \pi )^2\leq 3(r^2+ \lambda_k^2\gamma^2 r^{2\alpha}+\lambda_k^2)+\lambda_k^2\gamma^2 r^{2\alpha}\leq	
	\]
	\[
\leq 4\lambda_k^2(\frac{r^2}{\lambda_1^2}+ \gamma^2 r^{2\alpha}+1).
	\]
	Therefore, (\ref{BInt}) implies
\[
B_\alpha (\lambda_k, t)\geq \frac{\gamma \sin \alpha \pi}{4\lambda_k}\int\limits_0^\infty \frac{r^\alpha e^{-rT}}{\frac{r^2}{\lambda_1^2}+ \gamma^2 r^{2\alpha}+1} dr.
\]

It should be noted that the improper integral converges. Indeed,
\[
\int\limits_0^\infty \frac{r^\alpha e^{-rT}}{\frac{r^2}{\lambda_1^2}+ \gamma^2 r^{2\alpha}+1} dr\leq \int\limits_0^\infty r^\alpha e^{-rT}dr=	T^{-(\alpha+1)}\int\limits_0^\infty \tau^\alpha e^{-\tau}d\tau =	T^{-(\alpha+1)} \Gamma(\alpha+1).
\]	
	\end{proof}

\section{Forward	problem }

	Consider the following Cauchy problem:
		\begin{equation}\label{prob2a}
		\left\{
		\begin{aligned}
			&\partial_t v(t)  + (1+\gamma\, \partial_t^\alpha)A v(t) = f(t),\quad 0< t \leq T;\\
			&v(0) =\varphi,
		\end{aligned}
		\right.
	\end{equation}
where $\varphi$ is a given vector of $H$. We also call this problem \textit{the forward problem.}

\begin{thm}\label{Thprob2a}
	Let $\varphi\in H$ and $f(t) \in C ([0,T];D(A^\varepsilon))$ for some $\varepsilon\in (0,1)$. Then problem (\ref{prob2a}) has a unique solution and this solution has the representation
	\begin{equation}\label{v}
		v(t)= \sum\limits_{k=1}^\infty B_\alpha (\lambda_k, t) \varphi_k v_k+ \sum\limits_{k=1}^\infty
		\left[\int\limits_{0}^tB_\alpha (\lambda_k, t-\tau)f_k(\tau)d\tau\right] v_k.
	\end{equation}
	There are  constants $C$ and $ C_\varepsilon> 0 $ such that the
	following coercive type inequality holds:
	\begin{equation}\label{estimatev}
		||\partial_t v(t)||^2 + ||\partial_t^\alpha A v(t)||^2 \leq C t^{-2} ||\varphi||^2+ C_\varepsilon
		\max\limits_{t\in[0,T]} ||f||_\varepsilon^2,\quad 0 < t \leq T.
	\end{equation}
Moreover, there is a  constant $C_T$, depending on $T$, such that
\begin{equation}\label{estimatevA1}
	|| v(T)||_1^2 \leq C_T \big(||\varphi||^2+ 
	\max\limits_{t\in[0,T]} ||f||^2\big).
\end{equation}
\end{thm}

As noted in the introduction, the Cauchy problem (\ref{prob2a}) in the case when $A$ is the Laplace operator in the $N$-dimensional domain ($N=1,2,3$) has been considered by many authors. So the authors of \cite{Bazh} establish the Sobolev regularity of the homogeneous problem for both smooth and nonsmooth initial data $\varphi$. In particular, the authors proved an estimate similar to (\ref{estimatev}) (see estimate (2.16), note that in this paper  $f(t)\equiv 0$).

Formula (\ref{v}) for solving problem (\ref{prob2a}) is formally given in the same article of Bazhlekova, Jin, Lazarov, and Zhou \cite{Bazh}, but since this paper is devoted to the study of a homogeneous problem, the question of for which $\varphi$ and $f(t)$ it gives a solution to problem (\ref{prob2a}) is not considered. Note that functions $\varphi$ and $f(t)$ must be such that, for example, the series (\ref{v}) allows term-by-term differentiation with respect to the variable $t$.

The formula is also contained in papers \cite{Duc}, equation (2.2), \cite{Luc1}, equation (2.6) and \cite{Luc2}, equation (8). Again, since these works are devoted to the study of other problems, the above issues were not considered.

\begin{proof}According to the Fourier method, we expand the functions $\varphi$ and $f(t)$ into a Fourier series in terms of the system of eigenfunctions $\{v_k\}$ with coefficients $\varphi_k$ and $f_k(t)=(f(t), v_k)$ correspondingly. The solution of problem (\ref{prob1a}) will be sought in the form of a series $\sum_k T_k(t) v_k$ with unknown coefficients $T_k(t)$.
	
	It is easy to verify that the functions $T_k(t)$ are solutions of the Cauchy problem (\ref{BCauchyIN}) with the right-hand side $f(t)=f_k(t)$, the initial condition $y_0=\varphi_k$ and with the parameter $\lambda= \lambda_k$. Therefore (see (\ref{BCauchySolution}))
	\[
	T_k(t)=B_\alpha (\lambda_k, t) \, \varphi_k  + \int\limits_{0}^tB_\alpha (\lambda_k, t-\tau)f_k(\tau)d\tau.
	\]
	
	Thus, series (\ref{v}) is a formal solution to problem (\ref{prob2a}). It remains to show that the series itself and the series obtained after differentiation also converge. Let us get started with this task.
	
	By Parseval's equality, the first assertion of the Lemma \ref{Bazh}, and Holder's inequality, we have	
	\[
\bigg|\bigg|\sum\limits_{k=1}^j \left[\int\limits_{0}^t B_\alpha (\lambda_k, t-\tau)f_k(\tau)d\tau\right] v_k\bigg|\bigg|^2= \sum\limits_{k=1}^j \left[\int\limits_{0}^t B_\alpha (\lambda_k, t-\tau)f_k(\tau)d\tau\right]^2\leq
	\]
	\[
\leq \sum\limits_{k=1}^j \left[\int\limits_{0}^t |f_k(\tau)|d\tau\right]^2\leq C \max_{0\leq t\leq T}||f(t)||^2.	
	\]
	Same way,
	\[
	\bigg|\bigg|\sum\limits_{k=1}^j B_\alpha (\lambda_k, t)\varphi_k v_k\bigg|\bigg|^2\leq C ||\varphi||^2.	
	\]
	Now consider the series after differentiation. Apply Parseval's equality to obtain
	\[
\bigg|\bigg|\sum\limits_{k=1}^j \left[\partial_t\int\limits_{0}^t B_\alpha (\lambda_k, t-\tau)f_k(\tau)d\tau\right] v_k\bigg|\bigg|^2=\bigg|\bigg|\sum\limits_{k=1}^j \left[f_k(t)-\int\limits_{0}^t \partial_t B_\alpha (\lambda_k, t-\tau)f_k(\tau)d\tau\right] v_k\bigg|\bigg|^2\leq \]	
	\[
	\leq  2\max_{0\leq t\leq T}||f(t)||^2+ 2\sum\limits_{k=1}^j \left[\int\limits_{0}^t\partial_t B_\alpha (\lambda_k, t-\tau)f_k(\tau)d\tau\right]^2.
	\]
Generalized Minkowski inequality and	Lemma \ref{B'New_n} imply
\[
\sum\limits_{k=1}^j \left[\int\limits_{0}^t\partial_t B_\alpha (\lambda_k, t-\tau)f_k(\tau)d\tau\right]^2\leq C_n \bigg( \int\limits_0^t \tau^{\frac{1-\alpha}{n}-1} \bigg(\sum\limits_{k=1}^j |\lambda_k^{\frac{1}{n}} f_k(\tau)|^2\bigg)^{\frac{1}{2}}d\tau\bigg)^2\leq
\]
(choose $n$ such that $\frac{1}{n}<\varepsilon$)
\[
\leq C_\varepsilon \max_{t\in[0,T]} ||f(t)||_\varepsilon.
\]

Similarly, apply Parseval's equality and (\ref{Bn}) to get
	\[
\bigg|\bigg|\sum\limits_{k=1}^j \partial_t B_\alpha (\lambda_k, t)\varphi_k v_k\bigg|\bigg|^2\leq \frac{C}{t^2} ||\varphi||^2.	
\]

Equation (\ref{prob1a}) implies
\[
||(1+\gamma \partial_t^\alpha) A v||^2\leq ||\partial_t v||^2 + ||f(t)||^2.
\]
Therefore, it follows from the above
\[
\bigg|\bigg|(1+\gamma \partial_t^\alpha) A \sum\limits_{k=1}^j
\left[\int\limits_{0}^tB_\alpha (\lambda_k, t-\tau)f_k(\tau)d\tau\right] v_k\bigg|\bigg|^2\leq C_\varepsilon \max_{t\in[0,T]} ||f(t)||_\varepsilon,
\]
and
\[
\bigg|\bigg|(1+\gamma \partial_t^\alpha) A \sum\limits_{k=1}^j
B_\alpha (\lambda_k, t)\varphi_k v_k\bigg|\bigg|^2\leq C(\frac{1}{t^2} ||\varphi||+||f(t)||).
\]
Note that $||f(t)||\leq ||f(t)||_\varepsilon$, $\varepsilon>0$.

Thus, it has been shown that the function defined by the series (\ref{v}) is indeed a solution to problem (\ref{prob1a}).

Obviously, (\ref{estimatev}) follows from the established estimates. 

The uniqueness of the solution to problem (\ref{prob1a}) can be proved by the standard technique based on
completeness of the set of eigenfunctions $\{v_k\}$ in $H$. Indeed, suppose the problem has two solutions $v^1$ and $v^2$. Then the difference $v=v^1-v^2$ is a solution of the homogeneous problem:
\[
	\left\{
	\begin{aligned}
		&\partial_t v(t)  + (1+\gamma\, \partial_t^\alpha)A v(t) = 0,\quad 0< t \leq T;\\
		&v(0) =0.
	\end{aligned}
	\right.
\]
Let $v(t)$ be any soltion of this problem. Consider the Fourier coefficients
 $T_k(t)=(v(t), v_k)$. It is not hard to see, that $T_k$ is a solution of the Cauchy problem
 \[
 \partial_t T_k(t)  + \lambda_k (1+\gamma\, \partial_t^\alpha) T_k(t) = 0,\quad 0< t \leq T,\,\, T_k(0)=0.
 \]
 Corollary \ref{BCauchySolutionIN} implies that $T_k(t)\equiv 0$ for all $k$. Since 
the set of eigenfunctions $\{v_k\}$ complete in $H$, then $v(t)\equiv 0$.

It remains to prove the estimate (\ref{estimatevA1}). Let $S_j(t)$ be the partial sum of the series (\ref{v}). Then
\[
	||A S_j(T)||\leq  \bigg|\bigg|\sum\limits_{k=1}^j \lambda_k B_\alpha (\lambda_k, T) \varphi_k v_k\bigg|\bigg|+ \bigg|\bigg|\sum\limits_{k=1}^j
\left[\int\limits_{0}^T \lambda_k B_\alpha (\lambda_k, t-\tau)f_k(\tau)d\tau\right] v_k\bigg|\bigg|.
\]
Apply Parseval's equality and estimate (3) of Lemma \ref{Bazh} to obtain
\[
 \bigg|\bigg|\sum\limits_{k=1}^j \lambda_k B_\alpha (\lambda_k, T) \varphi_k v_k\bigg|\bigg|^2\leq C_T \sum\limits_{k=1}^j |\varphi|^2\leq C_T ||\varphi||^2.
\]
Similarly,
\[
\bigg|\bigg|\sum\limits_{k=1}^j
\left[\int\limits_{0}^T \lambda_k B_\alpha (\lambda_k, t-\tau)f_k(\tau)d\tau\right] v_k\bigg|\bigg|^2\leq C_T \sum\limits_{k=1}^j
\left[\int\limits_{0}^T \tau^{-\alpha}f_k(\tau)d\tau\right]^2\leq 
\]
(apply generalized Minkowski inequality)
\begin{equation}\label{estimate_f}
\leq C_T \bigg(\int\limits_{0}^T \tau^{-\alpha} \bigg(\sum\limits_{k=1}^j|f_k(\tau)|^2\bigg)^{1/2}d\tau\bigg)^2\leq C_T \max_{t\in[0,T]} ||f(t)||^2.
\end{equation}
The last two estimates imply (\ref{estimatevA1}).

The theorem is completely proven.

\end{proof}

\section{ Backward problem}
In the present paragraph we consider the following \textit{backward} problem:
\begin{equation}\label{probBackward}
	\left\{
	\begin{aligned}
		&\partial_t \omega(t)  + (1+\gamma\, \partial_t^\alpha)A \omega(t) = f(t),\quad 0< t < T;\\
		&\omega(T) =\psi,
	\end{aligned}
	\right.
\end{equation}
where $\psi$ is a given vector of $H$. It should be specially emphasized that for the backward problem the parameter $\gamma$ in the Rayleigh-Stokes equation plays an important role. If this parameter is equal to zero, then we get the classical diffusion equation, for which the backward problem is strongly ill-posed: not for any $\psi\in H$ (even not for any $\psi\in D(A^k), \, k\geq 1$) there is a solution, and if it exists, then a small change in $\psi$ leads to a very strong change in the solution (see e.g. Chapter 8.2 of \cite{Kaban}). In the next Theorem \ref{backward}, we show that the parameter $\gamma\neq 0$ $"$ennobles$"$ the backward problem for the Rayleigh-Stokes equation and its solution already exists for all $\psi\in D(A)$. 

Nevertheless problem (\ref{probBackward}) is also ill-posed in the sense of Hadamard
because of the same reason as the classical one ($ \gamma = 0 $): a
small variation of $u(T)$ in the norm of space $H$ may cause
arbitrarily large variations in the initial data. Indeed, let
$f(t)\equiv 0$ and take
\[
\omega(0)=\lambda^{-1+\varepsilon}_k\cdot\frac{1}{B_{\alpha}(\lambda_k,
	T)} \cdot v_k,\,\,\varepsilon>0,
\]
as a initial condition in problem (\ref{prob2a}).
Then the unique
solution of the problem (\ref{prob2a}) is (see Theorem \ref{prob2a})
\[
\omega(t)=\lambda^{-1+\varepsilon}_k\cdot\frac{B_{\alpha}(\lambda_k, t)}{B_{\alpha}(\lambda_k,
	T)} \cdot v_k.
\]
This function is a unique solution of problem (\ref{probBackward}) with $u(T)=\lambda^{-1+\varepsilon}_k v_k $.

Therefore, on the one hand, $||u(T)||=\lambda^{-1+\varepsilon}_k$
and it tends to zero as $k\rightarrow \infty$ (even
$||u(T)||_{a}\rightarrow 0$ for any $a<1-\varepsilon$), and on the
other,  according to Lemma \ref{BLower},
\[
||\omega(0)||\geq \lambda^{-1+\varepsilon}_k\cdot\frac{\lambda_k}{C(\alpha, \gamma, \lambda_1)} \rightarrow \infty \quad \text{when}
\quad k\rightarrow\infty.
\]

However, if we consider the norm of $u(T)$ in space $D(A)$, then
the situation will change completely; note the norm
$||u(T)||_1=\lambda_k^\varepsilon$ in this example is unbounded as
$k\rightarrow\infty$.

\begin{thm}\label{backward} Let $f(t)\in C([0,T];
	H)$. Then for any
	$\psi\in D(A)$ problem (\ref{prob2a}) has a unique solution.
	Moreover there exists a constant $C= C(\alpha, \gamma, \lambda_1, T)>0$, depending on $\alpha, \gamma, \lambda_1$  and $T$, such that
	\begin{equation}\label{backwardestimate_f}
		||\omega(t)||\leq C\,
		\big(||\omega(T)||_1+ \max\limits_{t\in[0,T]}|f(t)||\big).
	\end{equation}

	Let $f(t)\equiv 0$. Then there exist constants $C_1, C_2>0$, such that
	\begin{equation}\label{backwardestimate}
		C_1||\omega(0)||\leq
		||\omega(T)||_1\leq C_2||\omega(0)||.
	\end{equation}
\end{thm}

Note that in work \cite{AA} the baclward problem for the subdiffusion equation $\partial_t^\alpha u +Au =f$ was studied. It is shown that an estimate similar to (\ref{backwardestimate}) is valid only with the norm $||\omega(T)||_2$. However, if we take the derivative in the sense of Caputo in the subdiffusion equation, then the estimate (\ref{backwardestimate}) with the norm $||\omega(T)||_1$ is valid (see (\cite{Florida})).

\begin{proof} Let us take a solution (\ref{v}) with an unknown initial function $\varphi$ and use condition $\omega(T)=\psi$ to determine this unknown function. Then
	the Fourier coefficients of $\varphi$ has the form 
	\begin{equation}\label{coef}
	\varphi_k=\frac{1}{B_\alpha (\lambda_k, T)} \left(\psi_k -\int\limits_0^T B_\alpha (\lambda_k, t-\tau)f_k(\tau)d\tau \right),\,\, k\geq 1, 
	\end{equation}
	and the unique formal solution of the backward problem can be written as
\[	
\omega(t)= \sum\limits_{k=1}^\infty \frac{B_\alpha (\lambda_k, t)}{B_\alpha (\lambda_k, T)} \left(\psi_k -\int\limits_0^T B_\alpha (\lambda_k, t-\tau)f_k(\tau)d\tau \right)  v_k + \sum\limits_{k=1}^\infty
\left[\int\limits_{0}^tB_\alpha (\lambda_k, t-\tau)f_k(\tau)d\tau\right] v_k.
\]

According to Theorem \ref{Thprob2a}, this formal solution will indeed be a unique solution to the backward problem if the condition: $\varphi\in H$ is satisfied. Let us check this condition.

Apply Lemma \ref{BLower} to obtain
\[
\sum\limits_{k=1}^j \bigg|\frac{\psi_k}{B_\alpha (\lambda_k, T)}\bigg|^2\leq \frac{1}{C(\alpha, \gamma, \lambda_1)} \sum\limits_{k=1}^j|\lambda_k \psi_k|^2\leq \frac{1}{C(\alpha, \gamma, \lambda_1)}||\psi||_1^2.
\]
Again this lemma and estimate (3) of Lemma \ref{Bazh} imply (see (\ref{estimate_f}))
\[
\sum\limits_{k=1}^j \bigg|\frac{1}{B_\alpha (\lambda_k, T)}\bigg|^2 \left(\int\limits_0^T B_\alpha (\lambda_k, t-\tau)f_k(\tau)d\tau\right)^2\leq 
 \frac{C_T}{C(\alpha, \gamma, \lambda_1)}\max_{t\in[0,T]} ||f(t)||^2.
\]
Thus $\varphi=\sum_k \varphi_k v_k \in H$ and therefore, function $\omega(t)$ defined above is indeed a solution to the backward problem ((\ref{prob2a}).

Let us pass to the proof of estimate (\ref{backwardestimate_f}). We note right away that in order to estimate the norm in $H$, by virtue of the Parseval equality, it is sufficient to estimate a series of Fourier coefficients.

Application of Lemmas \ref{Bazh} and \ref{BLower} gives
\[
\sum\limits_{k=1}^j \left| \frac{B_\alpha (\lambda_k, t)}{B_\alpha (\lambda_k, T)} \psi_k \right|^2\leq \frac{1}{C(\alpha, \gamma, \lambda_1)}\sum\limits_{k=1}^j |\lambda_k \psi_k|^2\leq \frac{1}{C(\alpha, \gamma, \lambda_1)} ||\psi||^2.
\]

Apply Lemma \ref{BLower} for $B_\alpha$ in the denominator, estimate (1) of Lemma \ref{Bazh} for $B_\alpha$ in the numerator, and finally estimate (3) of Lemma \ref{Bazh} for $B_\alpha$ under the integral sign (see (\ref{estimate_f})). Then
\[
\sum\limits_{k=1}^j \bigg|\frac{B_\alpha (\lambda_k, t)}{B_\alpha (\lambda_k, T)}\bigg|^2 \left(\int\limits_0^T B_\alpha (\lambda_k, t-\tau)f_k(\tau)d\tau\right)^2\leq 
\frac{C_T}{C(\alpha, \gamma, \lambda_1)}\max_{t\in[0,T]} ||f(t)||^2.
\]
Similarly, by virtue of estimate (3) of Lemma \ref{Bazh} we obtain
\[
\sum\limits_{k=1}^j
\left[\int\limits_{0}^tB_\alpha (\lambda_k, t-\tau)f_k(\tau)d\tau\right]^2\leq C_T \max_{t\in[0,T]} ||f(t)||^2.
\]

As for estimate (\ref{backwardestimate}), its left part follows from the just proved estimate (\ref{backwardestimate_f}), while the right part is contained in (\ref{estimatevA1}).

\end{proof} 

As noted in the introduction, statements similar to Theorem \ref{backward} were known earlier (see, e.g., \cite{Luc1}, \cite{Luc2} and the bibliography therein). In these papers, operator $A$ is the Laplace operator with the Dirichlet condition in the domain $\Omega\subset R^N$, $N\leq 3$.

Let us make a few remarks:

1. The left side of the estimate (\ref{backwardestimate}) means that a small change of $\omega(T)$ in the norm $||\cdot||_1$ entails a small change in the norm $||\cdot||$ of the initial data.

2. The right side of this estimate asserts the unimprovability of the left side, i.e. it is impossible to replace the norm $||\cdot||_1$ with the norm $||\cdot||_{1-\varepsilon}$ with any $\varepsilon>0$.

3. For the classical diffusion equation (i.e. $\gamma=0$), the statement of Theorem \ref{backward} naturally does not hold, since $C(\alpha, 0, \lambda_1)=0$ (see the definition of this constant in Lemma \ref{BLower}). It should also be noted, that for the backward
problem for the classical diffusion equation (i.e. $\gamma =0$)
estimates of the type (\ref{backwardestimate}) on the scales of spaces $D(A^a)$
are generally impossible (see, for example, Chapter 8.2 of
\cite{Kaban}).

\section{Conditional stability}

Following work Luc, Huynh, O'Regan and Can \cite{Luc2}, in this section we consider the problem of conditional stability of the backward problem. In other words, we impose the following a priori bound condition on the initial data $\omega(0)=\varphi$:
\begin{equation}\label{bond}
||\varphi||^2_\varepsilon=\sum_{k=1}^\infty \lambda_k^{2\varepsilon} |\varphi_k|^2\leq \Phi^2_0, 
\end{equation} 
where $\varepsilon$ and $\Phi_0$ are positive constants, and consider a class of functions that satisfy this condition. The following statement is true:

\begin{thm}\label{stability} Let $\varphi\in D(A^\varepsilon)$ satisfy condition (\ref{bond}). Then there is a constant $C$ depending on $\alpha, \gamma, \lambda_1$ and $T$ such, that
	\begin{equation}\label{stabilityestimate}
	||\varphi||\leq C \big[ ||\psi|| +\max_{0\leq t\leq T} ||f(t)||\big]^{\frac{\varepsilon}{1+\varepsilon}} \Phi_0^{\frac{1}{1+\varepsilon}}.	
\end{equation}
\end{thm}

This theorem for the case of the Laplace operator with the Dirichlet condition in $N$, $N\leq 3$, - dimensional domain was proved in the above paper \cite{Luc2} (note that the condition of the theorem in \cite{Luc2} is slightly different). Note that in this case the estimates $\lambda_k\geq C k^{\frac{2}{N}}$ hold for the eigenvalues of the Laplace operator, and the proof in \cite{Luc2} is based on the convergence of the series 
\[
\sum_{k=1}^\infty\frac{1}{\lambda^2_k} \leq C\sum_{k=1}^\infty\frac{1}{k^{\frac{4}{N}}} < \infty.
\]

The proof of Theorem \ref{stability}  repeats the proof of \cite{Luc2} with a slight change. For the convenience of the reader, we present this proof.

We emphasize that the assertion of Theorem \ref{stability} is valid without any conditions on the spectrum of operator $A$.

\begin{proof}Set (see (\ref{coef}))
	\[
	\Phi_k=\psi_k -\int\limits_0^T B_\alpha (\lambda_k, t-\tau)f_k(\tau)d\tau, \, k\geq 1.
	\]
	Then
	\[
	||\varphi||^2=\sum_{k=1}^\infty\left|\frac{\Phi_k}{B_\alpha(\lambda_k,t)}\right|^2=
	\sum_{k=1}^\infty\frac{|\Phi_k|^{\frac{2\varepsilon}{1+\varepsilon}}|\Phi_k|^{\frac{2}{1+\varepsilon}}}{|B_\alpha(\lambda_k,t)|^2}.
	\]
	Apply the Holder inequality with parametrs $p=\frac{1+\varepsilon}{\varepsilon}$ and $q=1+\varepsilon$ to obtain
	\[
		||\varphi||^2\leq \left(\sum_{k=1}^\infty|\Phi_k|^2\right)^{\frac{\varepsilon}{1+
		\varepsilon}}\left(\sum_{k=1}^\infty\frac{1}{|B_\alpha(\lambda_k,t)|^{2\varepsilon}}\left|\frac{\Phi_k}{B_\alpha(\lambda_k,t)}\right|^2\right)^{\frac{1}{1+
		\varepsilon}}.
	\]
We use 	Lemma \ref{BLower} to get
	\[
	\sum_{k=1}^\infty\frac{1}{|B_\alpha(\lambda_k,t)|^{2\varepsilon}}\left|\frac{\Phi_k}{B_\alpha(\lambda_k,t)}\right|^2\leq \frac{1}{C^{2\varepsilon}(\alpha, \gamma,\lambda_1)}\sum_{k=1}^\infty|\lambda_k^\varepsilon\, \varphi_k|^2\leq \frac{\Phi_0^2}{C^{2\varepsilon}(\alpha, \gamma,\lambda_1)}.
	\]
	On the other hand  estimate (1) of Lemma \ref{Bazh} and the Holder inequality give
	\[
	\sum_{k=1}^\infty|\Phi_k|^2\leq 2||\psi||^2+2\sum_{k=1}^\infty \left[\int\limits_0^T |f_k(\tau)|d\tau\right]^2\leq 2||\psi||^2+2T\int\limits_0^T \sum_{k=1}^\infty |f_k(\tau)|^2 d\tau\leq
	\]
	\[
	\leq 2||\psi||^2+2T\max_{0\leq t\leq T}||f(t)||^2.
	\]
	The last two estimates imply estimate (\ref{bond}).
	\end{proof}

\section{Auxiliary	problem (\ref{prob1a})}

If we set $\varphi=0$, then from Theorem \ref{Thprob2a} we have the following result for our auxiliary problem (\ref{prob1a}).

\begin{thm}\label{Thprob1a}
	Let  $f(t) \in C ([0,T];D(A^\varepsilon))$ for some $\varepsilon\in (0,1)$. Then problem (\ref{prob1a}) has a unique solution and this solution has the representation
	\begin{equation}\label{v_fi0}
		v(t)= \sum\limits_{k=1}^\infty
		\left[\int\limits_{0}^tB_\alpha (\lambda_k, t-\tau)f_k(\tau)d\tau\right] v_k.
	\end{equation}
	Moreover, there is a constant $ C_\varepsilon> 0 $ such that the
	following coercive type inequality holds:
	\begin{equation}\label{estimatev_fi0}
		||\partial_t v(t)||^2 + ||\partial_t^\alpha A v(t)||^2 \leq C_\varepsilon
		\max\limits_{t\in[0,T]} ||f||_\varepsilon^2,\quad 0 \leq t \leq T.
	\end{equation}
\end{thm}

If $f(t)$ does not depend on $t$, then the statement of Theorem \ref{prob1a} is true for all $f \in H$.

\begin{cor}Let  $f \in H$. Then problem (\ref{prob1a}) has a unique solution and this solution has the representation
	\begin{equation}\label{vno_t}
		v(t)= \sum\limits_{k=1}^\infty
		f_k v_k \int\limits_{0}^t B_\alpha (\lambda_k, \tau)d\tau.
	\end{equation}
	Moreover, there is a constant $ C> 0 $ such that the
	following coercive type inequality holds:
	\begin{equation}\label{estimatev_no_t}
		||\partial_t v(t)||^2 + ||\partial_t^\alpha A v(t)||^2 \leq C
		||f||^2,\quad 0 < t \leq T.
	\end{equation}

\end{cor}

\begin{proof}Estimate (1) of Lemma \ref{Bazh} implies
	\[
\bigg|\bigg|\sum\limits_{k=1}^j
f_k v_k \partial_t\int\limits_{0}^t B_\alpha (\lambda_k, \tau)d\tau \bigg|\bigg|^2=\sum\limits_{k=1}^j
|f_k|^2|B_\alpha (\lambda_k, t)|^2\leq ||f||. 
	\]
	Using this estimate and repeating the arguments similar to the proof of Theorem \ref{prob1a},
	it is easy to check that (\ref{vno_t}) is indeed a unique solution to problem (\ref{prob1a}) and estimate (\ref{estimatev_no_t})
	holds true.

	\end{proof}

Note that the integral in (\ref{vno_t}) can be rewritten in the form
\[
\int\limits_0^t B_\alpha (\lambda_k, \tau) d\tau =\frac{1}{\lambda_k} (1- A_\alpha(\lambda_k, t)),
\]
where function $A_\alpha(\lambda_k, t))$ is also studied in \cite{Bazh}. In particular, it is shown that the Laplace transform of this function is
\[
\mathcal{L}\{A_\alpha(\lambda_k, \cdot)\}(z)=\frac{1+\gamma \lambda_k z^{\alpha-1}}{z+\gamma \lambda_k z^{\alpha}+\lambda_k},
\]
and the estimates
\[
0< A_\alpha(\lambda_k, t) \leq 1
\] 
hold.

\section{The second auxiliary and non-local	problems}

In this section we first consider auxiliary problem (\ref{prob1b}) and then we obtain the result for the non-local	problem (\ref{probMain}). 

\begin{thm}\label{Thprob1b}
	Let  $\psi\in H$. Then problem (\ref{prob1b}) has a unique solution and this solution has the representation
	\begin{equation}\label{w}
		w(t)= \sum\limits_{k=1}^\infty
		\frac{B_\alpha(\lambda_k, t)}{B_\alpha(\lambda_k, T)-1} \psi_k v_k.
	\end{equation}
	Moreover, there is a constant $ C> 0 $ such that the
	following coercive type inequality holds:
	\begin{equation}\label{estimatev_w}
		||\partial_t w(t)||^2 + ||\partial_t^\alpha A w(t)||^2 \leq C \frac{1}{t^2} ||\psi||^2,\quad 0 < t \leq T.
	\end{equation}
\end{thm}
\begin{proof}

The solution to problem (\ref{prob1b}) will be sought in the form of a Fourier series
\[
w(t)=\sum_{k=1}^\infty T_k(t) v_k,
\]
where $T_k(t)$ is the solution to the non-local problem
\begin{equation}\label{non-local_T}
	\left\{
	\begin{aligned}
		&\partial_t T(t)  + \lambda_k(1+\gamma\, \partial_t^\alpha)T_k(t) = 0,\quad 0< t < T;\\
		&T_k(T) =T_k(0)+ \psi_k,
	\end{aligned}
	\right.
\end{equation}
Assuming $T_k(0)= h_k$ to be known, we write a solution to equation (\ref{non-local_T}) with this initial condition (see (\ref{BazhEquationIN}))
\[
T_k(t)= h_k B_\alpha(\lambda_k, t).
\]
Now, using the non-local condition in (\ref{non-local_T}), we obtain an equation for finding the unknowns $h_k$:
\[
h_k (B_\alpha(\lambda_k, T)-1)=\psi_k.
\]
Since $T>0$ and $\lambda_k>0$, then $0<B_\alpha(\lambda_k, T)<1$ (see Lemma \ref{Bazh}). Therefore
\[
h_k=\frac{\psi_k}{B_\alpha(\lambda_k, T)-1}, \,\, |h_k|\leq C |\psi_k|,\,\, k\geq 1,
\]
and (\ref{w}) is a formal solution of problem (\ref{prob1b}).

The fact that the series (\ref{w}) converges is beyond doubt. It remains to show that this series can be term-by-term differentiated.

First apply Parseval's equality and (\ref{Bn}) to get
\[
\bigg|\bigg|\sum\limits_{k=1}^j \frac{\partial_t B_\alpha (\lambda_k, t)}{B_\alpha(\lambda_k, T)-1}\psi_k v_k\bigg|\bigg|^2\leq \frac{C}{t^2} ||\psi||^2, \,\, t>0.	
\]

On the other  hand equation (\ref{prob1b}) implies
\[
||(1+\gamma \partial_t^\alpha) A w||^2\leq ||\partial_t w||^2.
\]
Therefore, it follows from the above
\[
\bigg|\bigg|(1+\gamma \partial_t^\alpha) A \sum\limits_{k=1}^j\frac{
B_\alpha (\lambda_k, t)}{B_\alpha(\lambda_k, T)-1}\psi_k v_k\bigg|\bigg|^2\leq C\frac{1}{t^2} ||\psi||, \,\, t>0.
\]

The uniqueness of the problem solution (\ref{prob1b}) is based on
the completeness of the set of eigenfunctions $\{v_k\}$ in $H$ and the estimate $B_\alpha(\lambda_k, T)<1$. Indeed, if we assume that there are two solutions $w^1$ and $w^2$, then the difference $w=w^1-w^2$ is a solution to the homogeneous problem:
\[
\left\{
\begin{aligned}
	&\partial_t w(t)  + (1+\gamma\, \partial_t^\alpha)A w(t) = 0,\quad 0< t \leq T;\\
	&w(T)=w(0).
\end{aligned}
\right.
\]
Let $w(t)$ be any soltion of this problem. Consider the Fourier coefficients
$T_k(t)=(w(t), v_k)$. It is not hard to see, that $T_k$ is a solution of the problem
\[
\partial_t T_k(t)  + \lambda_k (1+\gamma\, \partial_t^\alpha) T_k(t) = 0,\quad 0< t \leq T,\,\, T_k(T)=T_k(0).
\]
The solution of the Cauchy problem with the initial data $T_k(0)=h_k$ has the form $T_k(t) = h_k B_\alpha(\lambda_k, t)$ (see Corollary \ref{BCauchySolutionIN}). The non-local condition implies $h_k B_\alpha(\lambda_k, T)=h_k$. Therefore $T_k(t)\equiv 0$ for all $k$. Since 
the set of eigenfunctions $\{v_k\}$ complete in $H$, then $w(t)\equiv 0$.

\end{proof}

Combining the statements of the last two theorems, we obtain the following assertion for the non-local problem (\ref{probMain}).

\begin{thm}\label{Umain} Let  $ \varphi \in H $ and $f(t) \in C ([0,T];D(A^\varepsilon))$ for some $\varepsilon\in (0,1)$. Then problem (\ref{probMain}) has a unique solution and this solution has the form \begin{equation}\label{u}
		u(t)= \sum\limits_{k=1}^\infty
		\left[\frac{\varphi_k-y_k(T)}{B_\alpha(\lambda_k, T)-1} B_\alpha(\lambda_k, t)  +y_k(t)\right] v_k,
	\end{equation}
where
\[
y_k(t)=\int\limits_{0}^tB_\alpha (\lambda_k, t-\tau)f_k(\tau)d\tau
\]
	Moreover, there are positive constants $C$ and $ C_\varepsilon$ such that the
	following coercive type inequality holds:
	\begin{equation}\label{estimateu}
		||\partial_t u(t)||^2 + ||\partial_t^\alpha A u(t)||^2 \leq C\frac{1}{t^2} ||\varphi||+ C_\varepsilon
		\max\limits_{t\in[0,T]} ||f||_\varepsilon^2,\quad 0 < t < T.
	\end{equation}
	
\end{thm}

\begin{proof}Let
	$ \varphi \in H $ and $f(t) \in C ([0,T];D(A^\varepsilon))$ for
	some $\varepsilon\in (0,1)$. As noted above, if we put
	$\psi=\varphi-v(T)\in H$ and $v(t)$ and $w(t)$ are the
	corresponding solutions of problems (\ref{prob1a}) and
	(\ref{prob1b}), then function $u(t)=v(t)+w(t)$ is a solution
	to problem (\ref{probMain}). Therefore, the solution of problem  (\ref{probMain}) has the form (\ref{w}).
	
	Estimate (\ref{estimateu}) follows from estimates (\ref{estimatev_fi0}) and (\ref{estimatev_w}). 
	
	\end{proof}

\section{Conclusion}
In many papers, explicit solutions of the simplest Rayleigh-Stokes problems are constructed. In \cite{Bazh}, the regularity of the solution was studied for the first time and a formal formula for solving an inhomogeneous problem in a three-dimensional domain was given. The backward problem has also been studied by the authors in domains with dimensions less than four. Some proofs of these results essentially use the fact that the dimension of the domain under consideration is $N\leq 3$.

In this paper, all these questions are investigated for the case of an abstract operator instead of the Laplace operator in the Rayleigh-Stokes problem. In addition, a new non-local problem for equation Rayleigh-Stokes is considered and it is shown that, unlike the backward problem, it is well-posed. As the operator $A$, we can take the Laplace operator with the Dirichlet condition in an arbitrary $N$-dimensional domain with a sufficiently smooth boundary.

\section{Acknowledgement}
The authors are grateful to Sh. A. Alimov for discussions of
these results.

The authors acknowledge financial support from the  Ministry of Innovative Development of the Republic of Uzbekistan, Grant No F-FA-2021-424.


\end{document}